\documentclass[12pt,a4paper]{amsart}
\usepackage{latexsym}
\usepackage{amsmath,amssymb,amsfonts,amsthm,amscd}
\usepackage[all,arc]{xy}
\usepackage{graphicx}
\usepackage{enumerate}
\usepackage{mathrsfs}
\usepackage{comment}
\usepackage{arydshln}
\usepackage{mathtools}
\usepackage{colonequals}
\usepackage{braket}

\newtheorem{thm}{Theorem}[section]

\newtheorem{lem}[thm]{Lemma}

\theoremstyle{definition}

\theoremstyle{remark}

\newcommand{\bA}{\mathbb{A}}
\newcommand{\bC}{\mathbb{C}}

\newcommand{\bQ}{\mathbb{Q}}
\newcommand{\bR}{\mathbb{R}}

\newcommand{\bZ}{\mathbb{Z}}
\newcommand{\cO}{\mathcal{O}}

\newcommand{\cZ}{\mathcal{Z}}
\newcommand{\sA}{\mathscr{A}}

\newcommand{\ol}{\overline}

\newcommand{\wt}{\widetilde}
\newcommand{\wh}{\widehat}

\newcommand{\Gm}{\mathbb{G}_{\mathrm{m}}}
\newcommand{\NT}{\mathrm{NT}}

\DeclareMathOperator{\Hom}{Hom}

\DeclareMathOperator{\End}{End}
\DeclareMathOperator{\Gal}{Gal}

\DeclareMathOperator{\Res}{Res}
\DeclareMathOperator{\Nm}{Nm}

\DeclareMathOperator{\Lie}{Lie}
\DeclareMathOperator{\Ch}{Ch}
\DeclareMathOperator{\Sh}{Sh}
\DeclareMathOperator{\GL}{GL}

\DeclareMathOperator{\GU}{GU}
\DeclareMathOperator{\oU}{U}
\DeclareMathOperator{\BC}{BC}
\DeclareMathOperator{\Ad}{Ad}
\DeclareMathOperator{\diag}{diag}

\DeclareMathOperator{\der}{d}
\DeclareMathOperator{\vol}{vol}

\title{Arithmetic Gan--Gross--Prasad conjecture for RSZ unitary Shimura curves}
\author{Yuta Nakayama}
\date{}
\begin{document}
\maketitle

\begin{abstract}
Xue proved an equational refinement of the unitary Shimura curve case of the arithmetic Gan--Gross--Prasad conjecture via the Gross--Zagier formula for quaternionic Shimura curves.
On the other hand, Rapoport, Smithling and Zhang posed a variant of the conjecture, using modified PEL type Shimura varieties, which we call RSZ Shimura varieties.
We reinterpret the result of Xue in terms of the modified Shimura curves.
We then use the reinterpretation to prove a case of the variant of the conjecture.
Our result combined with the work of Xue establishes a connection between the variant and the Gross--Zagier formula.
\end{abstract}

\setcounter{tocdepth}{1}
\tableofcontents

\section{Introduction} \label{sec:Intro}

We cherish Shimura varieties in number theory since they link moduli of abelian varieties and automorphic representation theory.
Along this line is, apart from the Langlands correspondence, a growing subject relating the intersection theory of special cycles on Shimura varieties and automorphic representation theory, following the Gross--Zagier formula \cite{GZ}.

The arithmetic Gan--Gross--Prasad conjecture \cite[\S 27]{GGPConj} is a typical example of the subject.
This AGGP conjecture relates a property of Chow groups of Shimura varieties as automorphic representations, the vanishing of certain Beilinson--Bloch height pairings on such varieties, and that of automorphic \(L\)-functions.
The conjecture for unitary groups was upgraded to a conjectural equation by Xue \cite[Conjecture 5.1]{XueAGGP}.
The paper proves the numerical conjecture for curves and for a certain case involving surfaces.
The curve case rests on Gross--Zagier formula for quaternionic Shimura curves \cite[Theorem 1.2]{YZZ}.

Unfortunately, the unitary Shimura varieties in these conjectures are not PEL but abelian type.
However, Rapoport, Smithling and Zhang \cite{RSZInt} invented PEL type variants of the varieties and corresponding conjectures.

In this paper, we aim to make explicit the relation expected in \cite[Remark 6.18]{RSZInt} between the generalized Gross--Zagier formula \cite[Theorem 1.2]{YZZ} and the modified AGGP conjecture \cite[Conjecture 6.12]{RSZInt} for curves.
As the work of Xue for curves establishes the relation betweeen the former and his numerical AGGP conjecture, here we achieve our goal by deriving a part of \cite[Conjecture 6.12]{RSZInt} from the numerical AGGP conjecture by Xue.
Our research is meaningful in that it is the first to explicitly discuss the RSZ variant of AGGP conjecture as far as the author knows.

We describe our results in more detail.
Set \(F\) to be a CM number field.
Its maximal totally real subfield is denoted by \(F_0\).
Choose a CM type \(\Phi\) of \(F/F_0\) and \(\varphi_0\in\Phi\).
Let \(W\) be a \(2\)-dimensional nondegenerate \(F/F_0\)-hermitian space such that the signature of \(W\otimes_{F,\varphi} \bC\) is \((0,2)\) if \(\varphi_0\neq\varphi\in\Phi\), and is \((1,1)\) if \(\varphi =\varphi_0\).
Choose a totally negative element \(u\) in \(W\).
Set \(W^\flat\) to be its orthogonal complement.

Set \(G = \oU(W)\) as an algebraic group over \(F_0\).
Also consider the torus \(Z^\bQ\) over \(\bQ\) as follows:
\[
    Z^\bQ(R) = \{z\in\Res_{F/\bQ}\Gm(R)\mid \Nm_{F/F_0}(z)\in R^{\times}\},
\]
where \(R\) is a \(\bQ\)-algebra.
Put \(\wt G = Z^\bQ\times \Res_{F_0/\bQ} G\).
We similarly define \(H\) and \(\wt H\) with respect to \(W^\flat\).
Also, put \(\wt{HG} \colonequals \wt H\times_{Z^\bQ} \wt G\).
We mainly consider in this paper compactifications \(\Sh (\wt {HG})\) of Shimura curves over the reflex field \(E\supseteq \varphi_0(F)\) with respect to the last group.
These Shimura curves are endowed with all the advantages mentioned in \cite[Remark 2.6]{RSZExp}.

We first focus on our numerical result.
Let \(\sA(\wt{HG})\) be the set of irreducible admissible representations of \(\wt{HG}(\bA_{\bQ,f})\) appearing in
\[
    H^1(\Sh (\wt {HG})(\bC),\bC))\colonequals \varinjlim_K H^1(\Sh_K (\wt {HG})(\bC),\bC).
\]
Let \(\pi\) be an irreducible tempered automorphic representation of \(\wt{HG}(\bA_\bQ)\) with trivial restriction to \(Z^\bQ(\bA_\bQ)\) such that \(\pi_f \in \sA(\wt{HG})\).
This \(\pi\) can be seen as an irreducible automorphic representation \(\pi_1\boxtimes\pi_2\) of \(H(\bA_{F_0})\times G(\bA_{F_0})\).
Under the decomposition \(\pi_f = \otimes_{v\nmid \infty}'\pi_v\) and that of its contragradient representation, take \(\phi = \otimes\phi_v\in\pi_f\) and \(\phi' = \otimes\phi'_v\in\wt\pi_f\) in the dual.
Choose a locally constant function \(\wt t_{\phi,\phi'}\colon\wt{HG}(\bA_{\bQ,f})\to\bC\) with compact support such that it induces \(\phi\otimes \phi'\) on \(\pi_f\) and zero maps on other representations in \(\sA(\wt{HG})\).
This \(\wt t_{\phi,\phi'}\) has an associated Hecke action \(T\) on the Jacobian \(J\) of \(\Sh_{\wt{K'}}(\wt{HG})\).

The notation \(z_{\wt{K'},0}\) means the cycle of \(\Sh_{\wt{K'}}(\wt{HG})\) defined by the diagonal embedding \(\Sh_{\wt{K'}\cap \wt H(\bA_\bQ,f)}(\wt{H})\to \Sh_{\wt{K'}}(\wt{HG})\), suitably normalized and cohomologically trivialized.
We think of it as an element of \(J(E)\otimes_\bZ \bC\).

\begin{thm}[Theorem \ref{eq}] \label{Intro-eq}
    We have for sufficiently small levels and for \(\beta\) and \(\alpha_v^\natural\) as in the last two paragraphs of \S \ref{subsec:Auto}
    \begin{align*}
        &\vol \Sh_{\wt{K'}}(\wt{HG})(\bC)\vol' \wt{K'}\langle T(z_{\wt{K'},0}), z_{\wt{K'},0}\rangle_\NT \\
        =& [E:\varphi_0(F)]\frac{\vol \wt K\sharp\Sh_{\wt K}(\wt H)(\bC)}{2^{\beta -1}}\frac{L_f'(1/2,\BC(\pi_1)\times \BC(\pi_2))}{L_f(1,\pi_1, \Ad)L_f(1,\pi_2, \Ad)}\prod_v \alpha_v^\natural(\phi_v, \phi'_v).
    \end{align*}
\end{thm}

Theorem \ref{Intro-eq} yields the modified AGGP conjecture in the following.
Set \(\cZ_{\wt{K'}, 0}\) to be the Hecke submodule of \(\Ch^1(\Sh_{\wt{K'}}(\wt{HG}))\otimes_\bZ \bC\) generated by \(z_{\wt{K'},0}\).
The equivalence below is a part of \cite[Conjecture 6.12]{RSZInt} for the curve case.
\begin{thm}[Theorem \ref{qual}]\label{Intro-AGGP}
    Assume that \(\pi_{2,\infty}\) is cohomological.
    For a sufficiently small level \(\wt{K'}\subseteq \wt{HG}(\bA_{\bQ,f})\), the following are equivalent.
    \begin{itemize}
        \item The pairing \(\langle*,z_{\wt{K'},0}\rangle_{\NT}\) is nonzero on \(\cZ_{\wt{K'}, 0}[\pi_f^{\wt{K'}}]\).
        \item The order of vanishing of \(L(s,\BC(\pi_1)\times \BC(\pi_2))\) at \(s = 1/2\) is \(1\) and \(\Hom_{\wt H(\bA_{\bQ,f})}(\pi_f, \bC)\) has dimension \(1\) with a generator nonvanishing on \(\pi_f^{\wt K'}\).
    \end{itemize}
\end{thm}

Finally, here is the organization of the paper.
In \S \ref{sec:Not}, we fix notations for the rest of the paper.
In \S \ref{sec:Num}, we deduce Theorem \ref{Intro-eq} from the curve case of \cite[Conjecture 5.1]{XueAGGP}.
In \S \ref{sec:AGGP}, we relate \S \ref{sec:Num} to \cite[Conjecture 6.12]{RSZInt}.

\subsection*{Acknowledgements}

It is a pleasure of the author to thank his advisor N.~Imai for the academic support.
The author also thanks H.~Xue for answering questions of the author in detail.

\section{Notations} \label{sec:Not}

We mostly use variants of notations in \cite{XueAGGP}.
For a finite set \(X\), its number of elements is denoted by \(\sharp X\).
For sets \(X\subseteq Y\), the characteristic function on \(Y\) with respect to \(X\) is denoted by \(\mathbf{1}_X\).
The unit of a group will be denoted by \(1\).
The Weil group of archimedean local fields are denoted by \(W_\bR\) and \(W_\bC\).
Let \(j\in W_\bR\backslash W_\bC\) be such that \(j^2 = -1\) and that \(jzj^{-1} = \ol z\) for \(z\in W_\bC\).
The notations \(\bA_L\) and \(\bA_{L,f}\) stand for the adele and the finite adele of a number field \(L\).
For a smooth representation \(\pi\) of a locally compact Hausdorff totally disconnected group \(G_0\), the contragradient representation is meant by \(\wt\pi\).
For the same \(G_0\), we set \(C_c^\infty(G_0)\) to be the set of \(\bC\)-valued locally constant functions with compact support.
If \(K\subseteq G_0\) is an open compact subgroup, then \(C_c^\infty(K\backslash G_0/K)\) will be the subset made of the \(K\)-bi-invariant functions.
Its analogue made of \(\bQ\)-valued functions will be denoted by \(C_c^\infty(K\backslash G_0/K, \bQ)\).

\subsection{Shimura curves} \label{subsec:Shi}
We comply with the notations in \cite[\S 3]{RSZInt}.
We first recall algebraic groups concerned.
Let \(F\) be a CM field with maximal totally real subfield \(F_0\).
Fix a CM type \(\Phi\) of \(F/F_0\) and its element \(\varphi_0\).

Take a nondegenerate \(F/F_0\)-hermitian space \(W\) of dimension \(2\).
Let the signature of \(W\otimes_{F,\varphi} \bC\) be \((0,2)\) or \((1,1)\) depending on whether \(\varphi_0\neq \varphi\in \Phi\) or \(\varphi = \varphi_0\).
Put \(G \colonequals \oU(W)\), a unitary group over \(F_0\).
We also define a \(\bQ\)-torus \(Z^\bQ\) and algebraic groups \(G^\bQ\) and \(\wt G\) over \(\bQ\) as
\begin{align*}
    Z^\bQ(R) &\colonequals \{z\in\Res_{F/\bQ}\Gm(R)\mid \Nm_{F/F_0}(z)\in R^{\times}\}, \\
    G^\bQ(R) &\colonequals \{g\in\Res_{F_0/\bQ}\GU (W)(R)\mid c(g)\in R^{\times}\},\\
	\wt G &\colonequals Z^\bQ\times_{\Gm}G^\bQ,
\end{align*}
\(R\) being a \(\bQ\)-algebra and \(c\) meaning the similitude.
We have \(\wt G\simeq Z^\bQ\times\Res_{F_0/\bQ} G\) by carrying \((z,g)\) to \((z,z^{-1}g)\).

We choose a totally negative vector \(u\in W\).
We define \(W^\flat\) as its orthogonal complement.
Groups \(H\), \(H^\bQ\) and \(\wt H\) are the same as \(G\), \(G^\bQ\) and \(\wt G\) except that we replace \(W\) with \(W^\flat\) in their definitions.
We also put
\[
\wt{HG} \colonequals \wt H\times_{Z^\bQ} \wt G\simeq Z^\bQ\times \Res_{F_0/\bQ}H\times \Res_{F_0/\bQ}G.
\]

Next, we recall the Shimura data in use below.
The map \(h_{Z^\bQ}\) in our Shimura datum \((Z^\bQ,\{h_{Z^\bQ}\})\) sends \(z\in\bC^\times\) to \((\ol z)\in Z^\bQ(\bR)\subseteq (\bC^\times)^\Phi\), where elements of \(\Phi\) induce the isomorphism \(F\otimes_\bQ \bR\simeq\bC^\Phi\) and then the inclusion.

To define a Shimura datum for \(G^\bQ\), we choose isomorphisms \(W\otimes_{F,\varphi} \bC\simeq \bC^2\) for each \(\varphi\in \Phi\) sending \(u\) to a multiple of \((0, 1)\) such that the hermitian forms on the left side becomes \(J_\varphi\colonequals \diag(2\delta_{\varphi,\varphi_0}-1, -1)\), where \(\delta\) is the Kronecker delta.
The left and right hand sides of the isomorphisms will be identified.
Take \(h_{G^\bQ}\colon\Res_{\bC/\bR}\Gm\to G^\bQ\) for our Shimura datum for \(G^\bQ\) so that the map which \(h_{G^\bQ}\) induces on \(\bR\)-valued points comes from various \(\bR\)-algebra maps \(\bC\to \End_\bC(W\otimes_{F,\varphi}\bC)\) for \(\varphi\in \Phi\) sending \(\sqrt{-1}\) to \(\sqrt{-1}J_\varphi\).

We similarly define \(h_{H^\bQ}\) using the above isomorphisms.
Take fiber products of these to define \(h_{\wt G}\), \(h_{\wt H}\) and \(h_{\wt{HG}}\), acquiring Shimura data with the same reflex field \(E\) associated with corresponding groups.
The field \(E\) is the composite of \(\varphi_0(F)\) and the reflex field of \(\Phi\).

Take projections of these three to define \(h_G\), \(h_H\) and \(h_{HG}\) that gives Shimura data with common reflex field \(\varphi_0(F)\) regarding \(\Res_{F_0/\bQ} G\), \(\Res_{F_0/\bQ} H\) and \(\Res_{F_0/\bQ} (H\times G)\).

We only consider sufficiently small levels for Shimura varieties.
Also, Shimura varieties coming from the above data are points or curves.
We mean smooth compactifications of such varieties by the symbols usually denoting the Shimura varieties themselves.
Moreover, in those symbols, we will not mention the conjugacy classes of the maps from Deligne torus.

\subsection{Volumes} \label{subsec:Vol}

Set \(\eta\colon \bA_{F_0}^\times/F_0^\times\to \{\pm 1\} \) to be the character in relation to \(F/F_0\) via global class field theory.
Fix a nontrivial additive character \(\psi\colon \bA_F/F\to \bC^\times\) such that \(\psi_v = \exp(2\pi\sqrt{-1}*)\) for each infinite place \(v\) of \(F\).
We put
\[
    \Delta_m\colonequals\prod_{i=1}^{m} L(i,\eta_v^i)
\]
for a positive integer \(m\).

For a place \(v\) of \(F_0\) and an \(F_v/F_{0,v}\)-hermitian space \(V\), the Haar measure \(\vol_v\) on \(\oU(V)\) will be the normalized local measure in \cite[\S 1.5]{XueAGGP}.
Namely, the pullback of the invariant differential form associated with \(\vol_v\) along the Cayley transform \(\Lie\oU(V)\to\oU(V)\) that carries \(X\) to \((1+X)(1-X)^{-1}\) is, at the origin, \(\Delta_{\dim V}\) times the one associated with the self-dual Haar measeure on \(\Lie\oU(V)\) with respect to \(\psi_v\).

For an \(F/F_0\)-hermitian space \(V\), we set
\[
    \vol\colonequals \Delta_{\dim V}^{-1} \prod_{v\nmid \infty}\vol_v,
\]
a Haar measure on \(\oU(V)(\bA_{F_0, f})\).
The symbol \(\vol\) will also denote the measure on \(Z^\bQ(\bA_{\bQ,f})\) such that \[
    \vol(Z^\bQ(\wh{\bZ})) = \sharp (\Sh_{Z^\bQ(\wh{\bZ})}(Z^\bQ)(\bC))^{-1},
\]
the product measure of these on \(\wt H(\bA_{\bQ,f}) = Z^\bQ(\bA_{\bQ,f})\times\oU(W^\flat)(\bA_{F_0,f})\), and the measure on Shimura varieties as in the next paragraph.
We fix three Haar measures all denoted by \(\vol'\) on \(Z^\bQ(\bA_{\bQ,f})\), \(H(\bA_{F_0,f})\times G(\bA_{F_0,f})\) and \(\wt{HG}(\bA_{\bQ,f})\) so that the third one is the product of the others.

In this paragraph, we decide a measure \(\vol\) on \(\Sh_{\wt{K'}}(\wt{HG})(\bC)\), compatible for different sufficiently small levels \(\wt{K'}\).
We only need to do it for \(\wt{K'} = \wt K\times K_G\) with \(\wt K\subseteq \wt{H}(\bA_{\bQ,f})\) and \(K_G\subseteq G(\bA_{\bQ,f})\).
Then our \(\vol\) is the product measure of the counting measure on \(\Sh_{\wt K}(\wt{H})(\bC)\) and the measure on \(\Sh_{K_G}(\Res_{F_0/\bQ} G)(\bC)\) in \cite[\S 4.1]{XueAGGP}\footnote{Multiply \(-1\) to the hermitian form of \(W\) to reach the case of \cite{XueAGGP}.} defined using the Hodge bundle on the curve.
See \cite[\S 3C]{LiuATL} for details of the last line bundle in the noncompact case.

\subsection{Cohomologically trivial cycles} \label{subsec:Cohtriv}

Generalizing \cite[\S 4.2]{XueAGGP}, we first define a cycle \(T(f)_K\) in \(\Sh_K(G_0)^2\) with complex coefficient for \(G_0\in\{\Res_{F_0/\bQ}(H\times G), Z^\bQ, \wt{HG}\}\), a level \(K\subseteq G_0(\bA_{\bQ,f})\) and a function \(f\) in \(C_c^\infty(K\backslash G_0(\bA_{\bQ,f})/K)\) as follows.
For \(x\in G_0(\bA_{\bQ,f})\), we have the image \(T(\mathbf{1}_{KxK})_K\) of \(\Sh_{K\cap xKx^{-1}}(G_0)\to\Sh_K(G_0)^2\), whose first projection is the transition morphism and whose second one is the right multiplication by \(x\) followed by the transition morphism.
Extend this construction linearly to the whole \(C_c^\infty(K\backslash G_0(\bA_{\bQ,f})/K)\).

Next, we recall the cohomologically trivial part of the Chow group.
For a smooth proper curve \(X\) over a number field \(L\subseteq \bC\), set \(J(X)\) to be its Jacobian variety.
The N\'eron--Tate height over \(L\) is meant by \(\langle *,*\rangle_\NT\).
We also set \(\Ch^1(X)_0\), the cohomologically trivial part, to be the kernel of the Betti cycle map \(\Ch^1(X)\to H^2(X(\bC), \bZ)\).
The last map can be written via the homomorphism \(H^1(X(\bC), \cO_{X(\bC)}^\times)\to H^2(X(\bC), \bZ)\) from the exponential sheaf sequence.
We have an injection \(\Ch^1(X)_0\to J(X)(L)\).

We specialize to the case \(X = \Sh_{\wt{K'}}(\wt{HG})\).
The Hecke algebra \(C_c^\infty(\wt{K'}\backslash \wt{HG}(\bA_{\bQ,f})/\wt{K'}, \bQ)\) acts on \(\Ch^1(X)\otimes_\bZ \bQ\).
Since the Hodge bundles at the end of \S \ref{subsec:Vol} are compatible with each other, their pull-back on each connecting component of \(X\) defines a Hecke equivariant retraction \(p\colon\Ch^1(X)\otimes_\bZ \bQ\to \Ch^1(X)_0\otimes_\bZ \bQ\).

\begin{lem}\label{project}
    The homomorphism \(R(f^-)\) in \cite[(6.13)]{RSZInt} coincides with \(p\) irrespective of the choice of \(f^-\).
\end{lem}

\begin{proof}
The exponential sheaf sequence of \(X(\bC)\) gives us an exact sequence
\[
    H^1(X(\bC), \bZ)\to N\to \Ch^1(X)_0\to 0
\]
for a Hecke submodule \(N\) of \(H^1(X(\bC), \cO_{X(\bC)})\) finitely generated as an abelian group.
In fact, we can embed not only \(N\) but \(N\otimes_\bZ \bQ\) into \(H^1(X(\bC), \cO_{X(\bC)})\).
The latter equals \(H^1(X(\bC), \bZ)\otimes_\bZ \bR\), a \(C_c^\infty(\wt{K'}\backslash \wt{HG}(\bA_{\bQ,f})/\wt{K'}, \bQ)\)-module that is semisimple by Matsushima's formula.
Therefore, \(\Ch^1(X)_0\otimes_\bZ \bQ\) is also a semisimple Hecke module.
In this situation, \(R(f^-)\) is unique by \cite[Remark 6.9(i)]{RSZInt}.
It equals \(p\) as the retraction to the same direct summand.
\end{proof}

We consider the diagonal embedding \(\wt H\to \wt{HG}\).
Put \(\wt K \colonequals \wt{K'}\cap \wt H(\bA_{\bQ,f})\).
The embedding \(\Sh_{\wt K}(\wt H)\to \Sh_{\wt{K'}}(\wt{HG})\) defines a cycle on the latter.
Let \(y_{\wt{K'}, 0}\in\Ch^1(\Sh_{\wt{K'}}(\wt{HG}))_0\otimes_\bZ \bQ\) be the image of this cycle by \(R(f^-)\).
Put
\[
z_{\wt{K'}, 0}\colonequals\vol (\wt K) y_{\wt{K'}, 0}\in\Ch^1(\Sh_{\wt{K'}}(\wt{HG}))_0\otimes_\bZ \bC.
\]
It will be identified with an element in \(J(\Sh_{\wt{K'}}(\wt{HG}))(E)\otimes_\bZ \bC\).
We have an analogue \(y_{K', 0}\) of \(y_{\wt{K'}, 0}\), denoted by \(y_{K, 0}\) in \cite[\S 5.1]{XueAGGP}, for the diagonal cycle \(\Sh_K(H)\subseteq \Sh_{K'}(\Res_{F_0/\bQ}(H\times G))\), where \(K = K'\cap H(\bA_{F_0,f})\).
We are using the notations closer to \cite[(6.14)]{RSZInt}.

\subsection{Automorphic representations} \label{subsec:Auto}

For a \(\bQ\)-algebraic group \(G_0\) in a Shimura datum in \S \ref{subsec:Shi}, put
\[
    H^m(\Sh (G_0)(\bC),\bC)\colonequals \varinjlim_K H^m(\Sh_K (G_0)(\bC),\bC),
\]
where \(m\) is the dimension of the Shimura varieties and \(K\) runs through compact open subgroups of \(G_0(\bA_{\bQ,f})\).
The symbol \(\sA(G_0)\) will denote the set of irreducible admissible representations of \(G_0(\bA_{\bQ,f})\) appearing in
\(H^m(\Sh (G_0)(\bC),\bC)\).

We take an irreducible tempered automorphic representation \(\pi\) of \(\wt{HG}(\bA_\bQ)\) whose finite part \(\pi_f\) is in \(\sA(\wt{HG})\) and whose restriction to \(Z^\bQ(\bA_\bQ)\) is trivial.
We see \(\pi\) also as an \(H(\bA_{F_0})\times G(\bA_{F_0})\)-representation.
We have decompositions \(\pi = \otimes_v'\pi_v\) and \(\wt\pi = \otimes_v'\wt{\pi_v}\), where the restricted tensor products are defined with respect to spherical elements whose pairings are \(1\) except at finitely many places.
We take \(\phi = \otimes \phi_v\in\pi_f = \otimes_{v\nmid \infty}'\pi_v\) and \(\phi' = \otimes \phi'_v\in\wt\pi_f = \otimes_{v\nmid \infty}'\wt{\pi_v}\).

We have a surjection
\[
    C_c^\infty(G_0(\bA_{\bQ,f}))\to\bigoplus_{\sigma\in\sA(G_0)}\End_\bC(\sigma)
\]
by the Hecke action using \(\vol'\).
If \(G_0 =\Res_{F_0/\bQ} (H\times G)\), then let \(t_{\phi,\phi'}\in C_c^\infty(G_0(\bA_{\bQ,f}))\) be an element sent by this morphism to \(\phi\otimes\phi'\) for \(\sigma = \pi_f\) and \(0\) otherwise.
If \(G_0 =Z^\bQ\), then let \(t\in C_c^\infty(G_0(\bA_{\bQ,f}))\) be an element that this morphism sends to the identity when \(\sigma\) is trivial and \(0\) otherwise.
If \(G_0 =\wt{HG}\), then let \(\wt t_{\phi,\phi'}\in C_c^\infty(G_0(\bA_{\bQ,f}))\) be an element which this morphism sends to \(\phi\otimes\phi'\) for \(\sigma = \pi_f\) and \(0\) otherwise.
For example, we could take \(\wt t_{\phi,\phi'} = t\otimes t_{\phi,\phi'}\).

Decompose \(\pi = \pi_1\boxtimes \pi_2\) as the outer tensor products of an \(H(\bA_{F_0})\)-representation and a \(G(\bA_{F_0})\)-representation.
We mean by \(\BC(\pi_i)\) the base change of each \(\pi_i\), a representation of \(\GL_i(\bA_F)\).
Let \(\beta\) be the sum of the number of irreducible cuspidal automorphic representations appearing in the expression of \(\BC(\pi_i)\) as isobaric sums for \(1 \leq i \leq 2\).
The function \(L(s,\BC(\pi_1)\times \BC(\pi_2))\) will mean the Rankin--Selberg convolution.
In general, let \(L_f\) be the partial \(L\)-function away from archimedean places.

Finally, we recall periods in \cite[\S 5.1]{XueAGGP} except a difference in the measures.
If \(v\) is a finite place of \(F_0\), then let
\[
    \alpha_v^\natural(\phi_v,\phi'_v)\colonequals \Delta_{2,v}^{-1}\frac{L(1,\pi_{1,v},\Ad)L(1,\pi_{2,v},\Ad)}{L(1/2,\BC(\pi_{1,v})\times\BC(\pi_{2,v}))}\int_{H(F_{0,v})} \phi'_v(\pi_v(h)(\phi_v))\der\vol_v(h),
\]
where \(h\) is diagonally embedded in \(H(F_{0,v})\times G(F_{0,v})\) and notations regarding \(L\)-functions mean local analogues of global notions.
If \(v\) is archimedean, then \(\phi_v\) and \(\phi'_v\) are not defined.
We formally put \(\alpha_v^\natural(\phi_v,\phi'_v)\colonequals \Delta_{2,v}^{-1}\vol_v(H(F_{0,v}))\), considering the period of the trivial representation.

\section{The numerical refinement of the AGGP conjecture} \label{sec:Num}

We prove Theorem \ref{Intro-eq} here from \cite[Conjecture 5.1]{XueAGGP} in the case of curves.

\begin{thm}\label{eq}
We have
\begin{align*}
    &\vol\Sh_{\wt{K'}}(\wt{HG})(\bC)\vol' \wt{K'}\langle T(\wt t_{\phi ,\phi'})_{\wt{K'},*}(z_{\wt{K'},0}), z_{\wt{K'},0}\rangle_\NT \\
    =&[E:\varphi_0(F)]\frac{\vol \wt K\sharp\Sh_{\wt K}(\wt H)(\bC)}{2^{\beta -1}}\frac{L_f'(1/2,\BC(\pi_1)\times \BC(\pi_2))}{L_f(1,\pi_1, \Ad)L_f(1,\pi_2, \Ad)}\prod_v \alpha_v^\natural(\phi_v, \phi'_v).
\end{align*}
\end{thm}

The left hand side is independent of \(\vol'\), the Hecke action depending on it.

\begin{proof}
The factor \(\vol\Sh_{\wt{K'}}(\wt{HG})(\bC)\vol' \wt{K'}T(\wt t_{\phi ,\phi'})_{\wt{K'},*}\) in the left hand side is independent of \(\wt{K'}\) and the choice of inverse image \(\wt t_{\phi,\phi'}\) similarly to the proof of \cite[Proposition A.4]{XueAGGP}.
Therefore we may assume that \(\wt{K'} = K_{Z^\bQ}\times K'\) for levels of \(Z^\bQ\) and \(\Res_{F_0/\bQ}(H\times G)\) and that \(\wt t_{\phi,\phi'} = t\otimes t_{\phi,\phi'}\).
Then
\[
\wt K = \wt{K'}\cap \wt H(\bA_{\bQ,f}) = K_{Z^\bQ}\times (K'\cap H(\bA_{F_0,f})) = K_{Z^\bQ}\times K
\]
and \(T(\wt t_{\phi,\phi'})_{\wt K'} = T(t)_{K_{Z^\bQ}}\times T(t_{\phi,\phi'})_{K'}\).
So, the left hand side of the theorem equals
\begin{align*}
    &(\vol\wt K)^2\vol\Sh_{\wt{K'}}(\wt{HG})(\bC)\vol' \wt{K'}\langle T(\wt t_{\phi ,\phi'})_{\wt{K'},*}(y_{\wt{K'},0}), y_{\wt{K'},0}\rangle_\NT \\
    =&\frac{(\vol K)^2\vol\Sh_{K'}(\Res_{F_0/\bQ}(H\times G))(\bC)}{\sharp\Sh_{K_{Z^\bQ}}(Z^\bQ)(\bC)}\vol' K'\vol'K_{Z^\bQ} \\
    &\sum_{y\in Z^\bQ(\bA_{\bQ,f})/K_{Z^\bQ}}t(y)\langle T(t_{\phi ,\phi'})_{K',*}(y_{K',0})\times \Sh_{K_{Z^\bQ}}(Z^\bQ), y_{K',0}\times \Sh_{K_{Z^\bQ}}(Z^\bQ)\rangle_\NT,
\end{align*}
where the last equality holds because \(T(\mathbf{1}_{K_{Z^\bQ}yK_{Z^\bQ}})_{K_{Z^\bQ},*}\) acts on the cycle \(\Sh_{K_{Z^\bQ}}(Z^\bQ)\) inside itself trivially for each \(y\).
The definition of \(t\) gives
\begin{align*}
    1 &= \int_{Z^\bQ(\bA_f)} t(y)\der\vol'(y) \\
    &=\vol'(K_{Z^\bQ})\sum_{y\in Z^\bQ(\bA_{\bQ,f})/K_{Z^\bQ}}t(y).
\end{align*}
Thus the left hand side of the theorem becomes by \cite[Appendix]{XueAGGP}\footnote{Multiply \(-1\) to the hermitian form of \(W\) as before. The appendix discusses the case of \cite[Conjecture 5.1]{XueAGGP} necessary for us with some assumptions, but the result without additional hypotheses is claimed in the paper.} for \(\pi_f\) tensored with the trivial representations at infinite places
\begin{align*}
    &[E:\varphi_0(F)](\vol K)^2\vol\Sh_{K'}(\Res_{F_0/\bQ}(H\times G))(\bC)\vol' K'\langle T(t_{\phi ,\phi'})_{K',*}(y_{K',0}), y_{K',0}\rangle_\NT \\
    &=[E:\varphi_0(F)]\frac{\vol K\sharp\Sh_K(H)(\bC)}{2^{\beta -1}}\frac{L_f'(1/2,\BC(\pi_1)\times \BC(\pi_2))}{L_f(1,\pi_1, \Ad)L_f(1,\pi_2, \Ad)}\prod_v \alpha_v^\natural(\phi_v, \phi'_v),
\end{align*}
where our \(\beta\) corresponds to the same symbol in \cite[\S 5.1]{XueAGGP} by strong multiplicity one, \(\Delta_2\) does not appear by our different choice of measures in defining \(\alpha_v^\natural\), and the N\'eron--Tate height pairing is now over \(\varphi_0(F)\).
The right hand side of this equation equals that of the theorem.
\end{proof}

\section{The AGGP conjecture} \label{sec:AGGP}

In this section, we relate the last section to Theorem \ref{Intro-AGGP}.

We list consequences of the assumption of the theorem that \(\pi_{2,\infty}\) is cohomological in the following sense.
If
\[
H^1(\Lie(\Res_{F_0/\bQ} G)\otimes_\bQ \bC, (\oU(1)\times \oU(1))\times \oU(2)^{[F_0:\bQ] - 1};\pi_{2,\infty})\neq 0,
\]
then there are only two possiblities for \(\pi_{2,\infty}\) according to \cite[Lemma D.2]{LiuFJCyc}.
They are trivial at all the archimedean places except the one below \(\varphi_0\).
At that place \(\pi_2\) is a holomorphic or antiholomorphic discrete series representation of \(\oU(1,1)\), and the lowest type with respect to \(\oU(1)\times \oU(1)\) sends \((s,t)\in\oU(1)\times \oU(1)\) to \(s/t\) or \(t/s\) respectively by shifting from differential forms in the Hodge decomposition of \(H^1(\Sh(G)(\bC), \bC)\) to modular or automorphic forms.

\begin{lem}\label{archrep}
    Let \(v\) be the place of \(F_0\) below \(\varphi_0\).
    The following is true under the assumption above.
    \begin{enumerate}
        \item\label{packet} The two possiblilities for \(\pi_{2,v}\) are in the same Vogan \(L\)-packet.
        They exhaust \(G(F_{0,v})\)-representations in the packet.
        \item\label{crucialL} We have
        \[
        L(s, \BC(\bC)\times \BC(\pi_{2,v})) = 4(2\pi)^{-2s-1}\Gamma(s+1/2)^2,
        \]
        \(\bC\) meaning the trivial representation and the latter \(\pi\) being the length of the circumference of the circle with diameter \(1\).
        \item\label{trivL} For archimedean places of \(F_0\) except \(v\), we have
        \[
        L(s, \BC(\bC)\times \BC(\bC)) = 4(2\pi)^{-2s-1}\Gamma(s+1/2)^2.
        \]
    \end{enumerate}
\end{lem}

\begin{proof}
(\ref{packet})
In \cite[\S 6]{GGPEx}, the direct sum of the two representations is the restriction of the representation of \(\GU(1,1) = (\GL_2(\bR)\times\bC^\times)/\bR^\times\) obtained by the discrete series representation \(D_{1,0}\) of \(\GL_2(\bR)\) with trivial central character and infinitesimal character \((1/2,-1/2)\) and the trivial \(\bC^\times\)-representation.
The two representations are in a Vogan \(L\)-packet as in the last line in the table on \cite[\S 6]{GGPEx}.

(\ref{crucialL})
Let
\[
    w\colonequals \begin{pmatrix}
    0 & 1 \\
    -1 & 0 \\
    \end{pmatrix}.
\]
We determine the \(L\)-parameter of \(\pi_{2,v}\) as in \cite[\S 6]{GGPEx}.
The real Langlands correspondence carries \(D_{1,0}\) to the induction \(N\) from \(W_\bC\) to \(W_\bR\) of the character \(z\mapsto z/|z|\).
The restriction \(N|_{W_\bC}\) is the sum of the last character and its inverse.
This restriction determines the \(L\)-parameter as \(W_\bC\ni z\mapsto (\diag(z/|z|, \ol z/|z|), 1)\) and \(j\mapsto (w, \ol *)\), where the complex conjugation acts on \(\GL_2(\bC)\) by the conjugation by \(w\) in \({}^L\oU(1,1) = \GL_2(\bC)\rtimes\Gal (\bC/\bR)\).

Similarly, the \(L\)-parameter of the trivial \(\pi_{1,v}\) is trivial on \(W_\bC\) and sends \(j\) to the complex conjugation.

We combine these parameters with the \(4\)-dimensional representaiton \(R\) of \({}^L(\oU(1)\times\oU(1,1))\) at the end of \cite[\S 22]{GGPConj}.
The induced representation has a basis comprising the following \(g_i\colon (\GL_1(\bC)\times \GL_2(\bC))\rtimes\Gal(\bC/\bR)\to\bC^2\) for \(1\leq i\leq 4\).
The function \(g_1\) (resp. \(g_2\), \(g_3\) and \(g_4\)) sends \((1,1,1)\) to \({}^t(1,0)\) (resp. \({}^t(0,1)\), \(0\) and \(0\)) and \((1,1,\ol *)\) to \(0\) (resp. \(0\), \({}^t(1,0)\) and \({}^t(0,1)\)).
The resulting representation \(W_\bR\to\GL(\bC^4)\) has the subrepresentations \(\bC g_1\oplus\bC g_4\) and \(\bC g_2\oplus\bC g_3\), both isomorphic to \(N\).
The result follows.

(\ref{trivL})
We only list main differences from (\ref{crucialL}).
As we expressed \(\GU(1,1)\) through \(\GL_2(\bR)\), we can express \(\GU(2)\) through the archimedean quaternion algebra.
We use the trivial representation of its unit group instead of \(D_{1,0}\) although the trivial representation is carried to \(D_{1,0}\) by the Jacquet--Langlands correspondence anyway.
\end{proof}

As in \cite[\S 6.4]{RSZInt}, the Hecke submodule of \(\Ch^1(\Sh_{\wt{K'}}(\wt{HG}))_0\otimes_\bZ \bC\) generated by \(z_{\wt{K'},0}\) is denoted by \(\cZ_{\wt{K'}, 0}\).
Let \(\cZ_{\wt{K'}, 0}[\pi_f^{\wt{K'}}]\) mean its \(\pi_f^{\wt{K'}}\)-typic part.

\begin{thm} \label{qual}
    Assume that
    \[
    H^1(\Lie(\Res_{F_0/\bQ} G)\otimes_\bQ \bC, (\oU(1)\times \oU(1))\times \oU(2)^{[F_0:\bQ] - 1};\pi_{2,\infty})\neq 0.
    \]
    The following are equivalent.
    \begin{enumerate}
        \item\label{pairing} The pairing \(\langle*,z_{\wt{K'},0}\rangle_{\NT}\) is nonzero on \(\cZ_{\wt{K'}, 0}[\pi_f^{\wt{K'}}]\).
        \item\label{L-function} The order of vanishing of \(L(s,\BC(\pi_1)\times \BC(\pi_2))\) at \(s = 1/2\) is \(1\) and \(\Hom_{\wt H(\bA_{\bQ,f})}(\pi_f, \bC)\) has dimension \(1\) with a generator nonvanishing on \(\pi_f^{\wt K'}\).
    \end{enumerate}
\end{thm}

\begin{proof}
We show that (\ref{L-function}) yields (\ref{pairing}).
The second condition of the former implies \(\Hom_{H(F_{0,v})}(\pi_v, \bC)\neq 0\) for each \(v\).
We show that for each finite place \(v\), there exist \(\phi_v\) and \(\phi'_v\) for which \(\alpha_v^\natural\) is nonzero.
If \(F\) is split over \(F_0\) at \(v\), then this problem is discussed in \cite[\S 1.4.1]{YZZ}.
Otherwise, \(H(F_{0,v})\) is compact and commutative so that \(\Hom_{H(F_{0,v})}(\pi_v, \bC) \simeq \Hom_{H(F_{0,v})}(\bC, \pi_v)\).
By sending \(1\) by a nonzero element of the latter, we get \(0\neq\phi_v\in \pi_v^{H(F_{0,v})}\).
Take an open subgroup \(K_v\subseteq G(F_{0,v})\) fixing \(\phi_v\).
Also take \(\phi'_v\in\Hom_\bC(\pi_v, \bC)\) such that \(\phi'_v(\phi_v)\neq 0\).
We may assume that \(\phi'_v\in \wt{\pi_v}\) by replacing it by the map carrying \(\phi_0\in \pi_v\) to
\[
    \int_{K_v} \phi'_v(g\phi_0)\der g,
\]
where \(\der\) is some Haar measure.
Our \(\phi_v\) and \(\phi'_v\) satisfy the required property.

The local periods do not vanish for spherical \(\phi_v\) and \(\phi'_v\) as in \cite[\S 5.1]{XueAGGP}.
Thus by the previous paragraph, there exist \(\phi = \otimes_v\phi_v\in \pi_f\) and \(\phi' = \otimes_v\phi'_v\in \wt\pi_f\) for which the product of \(\alpha_v^\natural\) is nonzero.
The assumption and Theorem \ref{eq} for these \(\phi\) and \(\phi'\) add up to \(\langle T(\wt t_{\phi ,\phi'})_{\wt{K'},*}(z_{\wt{K'},0}), z_{\wt{K'},0}\rangle_\NT\neq 0\).
The left term in the pairing is in \(\cZ_{\wt{K'}, 0}[\pi_f^{\wt{K'}}]\) by definition of \(\wt t_{\phi,\phi'}\) and the fact that \(\cZ_{\wt{K'}, 0}\) is semisimple by the proof of Lemma \ref{project}.

We now turn to the argument from (\ref{pairing}) to (\ref{L-function}).
The dimension of \(\Hom_{\wt H(\bA_{\bQ,f})}(\pi_f, \bC)\) is at most \(1\).
We will show that its generator is given by pairing with \(z_{\wt{K_1'},0}\) for various levels \(\wt{K_1'}\subseteq\wt{K'}\).
First, taking the pairing is well-defined for elements of
\[
    \cZ_\pi\colonequals \varinjlim_{\wt{K_1'}}\cZ_{\wt{K_1'}, 0}[\pi_f^{\wt{K_1'}}]
\]
since if \(\wt{K_2'}\subseteq\wt{K_1'}\) is another level and \(q\) is the corresponding transition morphism of Shimura curves of \(\wt{HG}\), then any \(y\in \cZ_{\wt{K_1'}, 0}[\pi_f^{\wt{K_1'}}]\) satisfies
\[
    \langle q^*y,z_{\wt{K_2'},0}\rangle_{\NT} = \langle y,\vol (\wt{K_2})q_*y_{\wt{K_2'},0}\rangle_{\NT} = \langle y,z_{\wt{K_1'},0}\rangle_{\NT},
\]
where \(\wt{K_2} = \wt{K_2'}\cap\wt H(\bA_{\bQ,f})\).
Next, we obtain the action of \(\wt{HG}(\bA_{\bQ,f})\) on \(\cZ_\pi\) from the Hecke action on \(\cZ_{\wt{K_1'}, 0}[\pi_f^{\wt{K_1'}}]\).
Since \(\cZ_{\wt{K'}, 0}\) is a cyclic Hecke module, \(\cZ_{\wt{K'}, 0}[\pi_f^{\wt{K'}}]\simeq\pi_f^{\wt{K'}}\) by assumption so that \(\cZ_\pi\simeq\pi_f\).
Thus the pairing defines the desired generator of \(\Hom_{\wt H(\bA_{\bQ,f})}(\pi_f, \bC)\) by the assumption since \(z_{\wt{K_1'},0}\) is unchanged by the action of \(\wt H(\bA_{\bQ,f})\).

We have \(\pi_{1,f}\in\sA(\Res_{F_/\bQ} H)\) and \(\pi_{2,f}\in\sA(\Res_{F_/\bQ} G)\).
Then the infinite part \(\pi_{1,\infty}\) is trivial by the real approximation.
Also, \(\pi_{2,\infty}\) is as in the beginning of this section.
As \(T(\wt t_{\phi ,\phi'})_{\wt{K'},*}(z_{\wt{K'},0})\) for various \(\phi\) and \(\phi'\) generate \(\cZ_{\wt{K'}, 0}[\pi_f^{\wt{K'}}]\), Theorem \ref{eq} shows that the finite part of the \(L\)-function has order not more than \(1\) at \(1/2\).
Lemma \ref{archrep} shows that the infinite part has order \(0\) there.
We show that the \(L\)-function vanishes there relying on the local Gan--Gross--Prasad conjecture.

Take the relevant quasi-split inner form \(H^0 \times G^0\) of \(H\times G\) in the sense of \cite[\S 2]{GGPConj}.
For each place \(v\) of \(F_0\), the Vogan packet of \(H^0_v \times G^0_v\) to which \(\pi_{1,v}\boxtimes\pi_{2,v}\) belongs contains a unique representation \(\pi^1_v\) of a relevant pure inner form \(H^1_v\times G^1_v\) of \(H^0_v\times G^0_v\) such that \(\Hom_{H^1_v}(\pi^1_v,\bC)\neq 0\) by \cite[\S 10]{GGPEx}.
Similarly to \cite[\S 26(1)]{GGPConj}, whether \(H^1_v\times G^1_v\) for various \(v\) are coherent is judged by \(\epsilon (1/2, \BC(\pi_1)\times\BC(\pi_2))\).
For \(v\) not below \(\varphi_0\), we have \(H^1_v = H(F_{0,v})\), \(G^1_v = G(F_{0,v})\) and \(\pi^1_v = \pi_{1,v}\boxtimes\pi_{2,v}\) by the already seen latter half of (\ref{L-function}) and the archimedean description of \(\pi\).
For \(v\) below \(\varphi_0\), we find that \(G^1_v\neq G(F_{0,v})\) by the same description and Lemma \ref{archrep} (\ref{packet}).
Thus \(H^1_v\times G^1_v\) for various \(v\) are incoherent, and \(\epsilon (1/2, \BC(\pi_1)\times\BC(\pi_2)) = -1\).
This results in the vanishing of the \(L\)-function.
\end{proof}

\bibliographystyle{amsplain}

\noindent
Yuta Nakayama\\
Graduate School of Mathematical Sciences, The University of Tokyo,
3-8-1 Komaba, Meguro-ku, Tokyo, 153-8914, Japan \\
nkym@ms.u-tokyo.ac.jp
\end{document}